\documentclass[11pt,reqno,british,a4paper,twoside]{article}

\usepackage[utf8]{inputenc}
\usepackage[T1]{fontenc}
\usepackage[main=british]{babel}
\usepackage{mathtools,amssymb,amsthm}
\usepackage{textcomp,enumitem,url}
\usepackage[babel=true]{microtype}

\setlist{noitemsep}
\setlist[enumerate,1]{label=\textnormal{(\roman*)}}

\newtheorem{theorem}{Theorem}[section]
\newtheorem{proposition}[theorem]{Proposition}
\newtheorem{lemma}[theorem]{Lemma}

\theoremstyle{definition}
\newtheorem{definition}[theorem]{Definition}

\theoremstyle{remark}
\newtheorem{remark}[theorem]{Remark}
\newtheorem{example}[theorem]{Example}

\numberwithin{equation}{section}

\newcommand{\Lie}[1]{\mathrm{#1}}
\newcommand{\lie}[1]{\operatorname{\mathfrak{#1}}}
\newcommand{\SO}{\Lie{SO}}
\newcommand{\Un}{\Lie{U}}
\newcommand{\so}{\lie{so}}
\newcommand{\un}{\lie{u}}

\newcommand{\colvector}[1]{#1}
\newcommand{\va}{\colvector{a}}
\newcommand{\vb}{\colvector{b}}
\newcommand{\vc}{\colvector{c}}
\newcommand{\vd}{\colvector{d}}
\newcommand{\ve}{\colvector{e}}
\newcommand{\vf}{\colvector{f}}
\newcommand{\vg}{\colvector{g}}
\newcommand{\vh}{\colvector{h}}
\newcommand{\vk}{\colvector{k}}
\newcommand{\vl}{\colvector{\ell}}
\newcommand{\vm}{\colvector{m}}
\newcommand{\vp}{\colvector{p}}
\newcommand{\vq}{\colvector{q}}
\newcommand{\vu}{\colvector{u}}
\newcommand{\vv}{\colvector{v}}
\newcommand{\vA}{\colvector{A}}
\newcommand{\vB}{\colvector{B}}
\newcommand{\vP}{\colvector{P}}
\newcommand{\vQ}{\colvector{Q}}

\newcommand{\hook}{\mathbin{\lrcorner}}

\newcommand{\bR}{\mathbb{R}}
\newcommand{\bC}{\mathbb{C}}
\newcommand{\GG}{\mathbf{G}}
\newcommand{\ii}{\mathbf{i}}

\newcommand{\CH}{\bC H}

\newcommand{\Mu}{\mathrm{M}}

\newcommand{\LC}{{\mathrm{LC}}}
\newcommand{\LCS}{{\mathrm{LC}(S)}}

\newcommand{\id}{\mathop{}\!\mathrm{id}}

\newcommand{\any}{\,\cdot\,}

\newcommand{\tallmstrut}{\rule{0pt}{2.7ex}}

\newcommand{\eqbreak}[1][2]{\\&\hspace{#1em}}
\newcommand{\eqand}[1][1]{\hspace{#1em}\text{and}\hspace{#1em}}

\begin{document}

\title{\bfseries The c-map on groups}

\author{Oscar Macia and Andrew Swann}

\date{}
\maketitle

\begin{abstract}
  We study the projective special Kähler condition on groups,
  providing an intrinsic definition of homogeneous projective special
  Kähler that includes the previously known examples.
  We give intrinsic defining equations that may be used without
  resorting to computations in the special cone, and emphasise certain
  associated integrability equations.
  The definition is shown to have the property that the image of such
  structures under the c-map is necessarily a left-invariant
  quaternionic Kähler structure on a Lie group.
\end{abstract}

\section{Introduction}

In the search for manifolds with special, or even exceptional,
holonomy developments in theoretical physics have provided a fruitful
ground for examples.
In particular, the study of T-duality between type IIA and type IIB
superstring theories from the point of view of the low energy
effective Lagrangians for \( D=4 \), \( N=2 \) supergravity has given
insight in the relation between Kähler geometry and hyperKähler or
quaternionic Kähler geometry, through a mechanism known as the c-map
originally introduced in Cecotti, Ferrara,
Girardello~\cite{Cecotti-FG:II}, Ferrara and Sabharwal~\cite{Ferrara-S:q}.

For supersymmetric field theories without gravity, supersymmetry is
regarded as a global symmetry and the moduli space of scalar fields of
vector multiplets is (affine) special Kähler
\cite{Craps-RTvP:what-sK,%
Freed:special,Lledo-MvPV:special,Strominger:special}, while the
geometry of the moduli space of scalar fields in the hypermultiplets
is of hyperKähler type.
When supersymmetry is imposed as a local symmetry, thus in the context
of supergravity, the geometries of the above moduli spaces of scalar
fields become projective special Kähler and quaternionic Kähler,
respectively, see~\cite{Alvarez-Gaume-F:geometrical,%
Bagger-W:couplings,Castellani-DF:sp-wo,dAuria-FF:special,%
DeWit-vP:potentials}.
The quaternionic Kähler nature of the hypermultiplet metric was first
described in Ferrara and Sabharwal~\cite{Ferrara-S:q}.
In this context the ``rigid c-map'' associates to every special Kähler
manifold of complex dimension \( n \) a dual hyperKähler manifold of
quaternionic dimension~\( n \), and the ``local c-map'' associates a
quaternionic Kähler manifold of quaternionic dimension \( n+1 \) to
each projective special Kähler manifold of complex dimension~\( n \).

Although the (local) c-map has its origins in supergravity, it has
substantial mathematical interest through the work of de Wit and
Van~Proeyen~\cite{DeWit-vP:special} where it was used to correct
Alekseevsky's classification~\cite{Alekseevsky:solvable} of
quaternionic Kähler manifolds admitting a transitive completely
solvable group of isometries.
Recently, mathematical descriptions of the c-map in general have been
given in~\cite{Alekseevsky-CM:conification} and~\cite{Macia-S:c-map}.
The former shows that the local formulas derived by Ferrara \&
Sabharwal~\cite{Ferrara-S:q} are indeed obtainable by appropriate
conification procedures; the latter provides a geometric approach to
the global geometry of the c-map via the twist construction and
elementary deformations.

The first applications of the c-map were to group manifolds, and
papers such as
\cite{Cecotti-FG:II,DeWit-vP:special,DeWit-VvP:symmetry} provide
several tables of resulting structures.
However, the precise mathematical motivation for the classes of
examples covered remains unclear, and from a mathematical point of
view assumptions derived from supergravity may not necessarily be
relevant for the mathematical applications.
Indeed all groups obtained are completely solvable, but it is an open
conjecture of Alekseevsky whether all homogeneous quaternionic Kähler
metrics of negative scalar curvature are left-invariant structures on
completely solvable groups.

The purpose of the current paper is to provide a first step towards
understanding what constraints the geometric c-map
in~\cite{Macia-S:c-map} may impose.
The initial data is a group manifold carrying an invariant projective
special Kähler structure.
However, traditionally the definition of projective special Kähler
\cite{Freed:special} is specified via the geometry of an auxiliary
cone, rather than intrinsically, and is not immediately clear which
structures should be regarded as homogeneous.
We thus start with a left-invariant Kähler structure on a Lie
group~\( S \), and work through the conditions that this admits a
projective special Kähler structure.
In the first instance we pass to the cone~\( C \) and study the
standard equations there.
We find a certain of integrability conditions enable one to quickly
get certain results about projective special Kähler manifolds that are
Kähler products.
Thereafter we show that the assumption that the Kähler form of \( S \)
is exact ensures the defining objects descend well from \( C \)
to~\( S \), and give us both a reasonable definition of homogeneous
structure and an intrinsic formulation of the projective special
Kähler condition directly on~\( S \).
We illustrate how these equations and the associated integrability
conditions may be used in a four-dimensional example.
Finally, we demonstrate the reasonableness of our definition by
proving that the c-map applied to a homogeneous structure on a group
always yields a group manifold with left-invariant quaternionic Kähler
structure.

While the focus of this paper is on group manifolds, its worth noting
that mathematically the c-map on inhomogeneous data is known to
construct previously unknown complete inhomogeneous quaternionic
Kähler~\cite{Cortes-DJL:cohom-1,Cortes-DL:psr,%
Cortes-DS:completeness-survey}.

As we were finalising this manuscript, Mauro Mantegazza kindly sent us
a copy of~\cite{Mantegazza:psK}.
There he obtains the intrinsic equations for projective special Kähler
manifolds in general, even when the Kähler form is not exact, and
provides various global results.
That paper also uses the characterisation to show that the homogeneous
examples in real dimension four are exactly the two cases we
consider in this paper, and in particular the exactness condition is
necessarily satisfied.

\section{The special Kähler conditions}
\label{sec:sK-cond}

Projective special Kähler manifolds~\( S \) are best defined and
understood via their cones~\( C \), cf.~\cite{Freed:special}.
In this section, we will start with a left-invariant Kähler structure
on a group manifold~\( S \) and use the associated cone to derive the
relevant equations in a left-invariant frame.
This will follow the general picture described
in~\cite{Macia-S:c-map}.

Suppose \( S \) is a Lie group with Lie algebra~\( \lie{s} \), and
that this Lie group carries a left-invariant Kähler structure with
complex structure~\( J \), metric~\( g_{S} \) and Kähler
form~\( \omega_{S} = g_{S}(J\any,\any) \).
Choose an orthonormal basis
\( \{X_{1},\dots,X_{n},JX_{1},\) \(\dots, JX_{n}\} \)
for~\( \mathfrak{s} \), and write \( A_{i} = X_{i} \),
\( B_{i} = JX_{i} \), for \( i = 1,\dots,n \).
Denote by \( \{a^{i},b^{i}: 1\leqslant i\leqslant n\} \) the
corresponding dual basis of left-invariant one-forms.
The complex structure acts on \( \mathfrak{s}^{*} \) with
\( Ja^{i} = b^{i} \).

In what follows it will be often useful to resort to matrix notation.
Therefore we introduce \( \bR^{n} \)-valued one-forms
\( \va = (a^{i}) \), \( \vb = (b^{i}) \), and the
\( \bR^{2n} \)-valued coframe \( \theta = (\va,\vb) \).
The metric and Kähler forms are
\begin{equation*}
  g_{S} = \theta^{T}\theta = \va^{T}\va +\vb^{T}\vb
  \eqand
  \omega_{S} = \frac{1}{2} \theta^{T}\wedge J\theta
  = \va^{T}\wedge\vb.
\end{equation*}
The connection one-form \( \omega_{\LCS} \) of the Levi-Civita
connection of~\( g_{S} \) is the skew-symmetric matrix determined by
the structural equations
\begin{equation*}
  d\theta = -\omega_{\LCS}\wedge\theta,
\end{equation*}
As the structure is Kähler, we have that \( \omega_{\LCS} \) takes
values in \( \un(n) \leqslant \so(2n) \), so we may write
\begin{equation*}
  \omega_{\LCS} =
  \begin{pmatrix}
    \mu&\lambda \\
    -\lambda&\mu
  \end{pmatrix}
  ,
\end{equation*}
with \( \mu = (\mu^{i}_{j}) \), \( \lambda = (\lambda^{i}_{j}) \)
\( n\times n \)-matrices of one-forms satisfying the following
symmetries
\begin{equation*}
  \mu^{T} = -\mu,\quad \lambda^{T} = \lambda.
\end{equation*}
The structural equation is thus
\begin{equation}
  \label{eq:structural}
  d
  \begin{pmatrix}
    \va \\
    \vb
  \end{pmatrix}
  = -
  \begin{pmatrix}
    \mu & \lambda \\
    -\lambda & \mu
  \end{pmatrix}
  \wedge
  \begin{pmatrix}
    \va \\
    \vb
  \end{pmatrix}
  .
\end{equation}
The curvature of~\( S \) now has the form
\begin{equation*}
  \Omega_{S}
  = d\omega_{\LCS} + \omega_{\LCS}\wedge \omega_{\LCS}
  =
  \begin{pmatrix}
    \Mu & \Lambda \\
    -\Lambda & \Mu
  \end{pmatrix}
  ,
\end{equation*}
where
\begin{equation}
  \label{eq:Mu-Lambda}
  \Mu = d\mu + {\mu}\wedge{\mu} - {\lambda}\wedge{\lambda}
  \eqand
  \Lambda= d \lambda + \mu\wedge\lambda + \lambda\wedge\mu.
\end{equation}

To introduce the projective special Kähler conditions, we need to
assume that \( (S,2\omega_{S}) \) is Hodge, meaning that there is a
circle bundle \( \pi\colon C_{0}\to S \) with connection one-form
\( \varphi \) such that
\begin{equation*}
  d\varphi = 2\pi^{*}\omega_{S}.
\end{equation*}
We write \( \pi^{*}\omega_{S} = \tilde{\va}^{T}\wedge\tilde{\vb} \),
where \( \tilde{\va} = \pi^{*}\va \), \( \tilde{\vb} = \pi^{*}\vb \),
and let \( X \) be the vector field generating the circle action in
the fibres.
In particular, \( \varphi(X) = 1 \), \( L_{X}\varphi = 0 \), and
\( \pi_{*}X = 0 \).
Then, the (complex) cone \( C \) over \( S \) is defined to be
\begin{equation*}
  C = \bR_{>0}\times C_{0}.
\end{equation*}

Let \( t \) be the standard coordinate on \( \bR_{>0} \) and put
\( \hat{\va} = t\tilde{\va} \), \( \hat{\vb} = t\tilde{\vb} \),
\( \hat{\varphi} = t\varphi \), \( \hat{\psi} = dt \).
Then \( C \) carries a pseudo-Kähler structure with metric and Kähler
form given by
\begin{equation*}
  g_{C} = \hat{\va}^{T}\hat{\va} + \hat{\vb}^{T}\hat{\vb} - \hat{\varphi}^{2} -
  \hat{\psi}^{2},\qquad
  \omega_{C} = \hat{\va}^{T}\wedge\hat{\vb}-\hat{\varphi}\wedge\hat{\psi}.
\end{equation*}
We denote its complex structure by \( J \) and note that the conic
symmetry \( X \) satisfies \( JX = t\partial_{t} \) and
\( g_{C}(X,X) = -t^{2} \).

Then \( (\hat{\va},\hat{\vb},\hat{\varphi},\hat{\psi}) \) is a unitary
coframe for~\( C \) and we put
\( \theta_{C} = (\hat{\va},\hat{\vb},\hat{\varphi},\hat{\psi})^{T} \).
The Levi-Civita connection of \( C \) is uniquely determined by
\( d \theta_{C} = -\omega_{\LC}\wedge \theta_{C} \) together with
\( \omega_{\LC}^{T}\GG + \GG\omega_{\LC}^{} = 0 \) and
\( \ii\omega_{\LC} = \omega_{\LC}\ii \), where
\begin{equation*}
  \GG =
  \begin{pmatrix}
    \id_{2n}&0 \\
    0 & -\id_{2}
  \end{pmatrix}
  \eqand
  \ii =
  \begin{pmatrix}
    \ii_{2n} & 0 \\
    0 & \ii_{2}
  \end{pmatrix}
  =
  \begin{pmatrix}
    0 & \id_{n} & 0 & 0\\
    -\id_{n} & 0 & 0 & 0 \\
    0 & 0 & 0 & 1\\
    0 & 0 & -1 & 0
  \end{pmatrix}.
\end{equation*}
Using
\begin{gather*}
  d\hat{\psi} = 0,\quad d\hat{\varphi} =
  \frac{1}{t}(\hat{\psi}\wedge\hat{\varphi} +
  2\hat{\va}^{T}\wedge\hat{\vb}),\\[1ex]
  d
  \begin{pmatrix}
    \hat{\va}\\
    \hat{\vb}
  \end{pmatrix}
  = \frac{1}{t}
  \begin{pmatrix}
    \hat{\psi}\id_{n}-\hat{\mu}&-\hat{\lambda}\\
    \hat{\lambda}&\hat{\psi}\id_{n}-\hat{\mu}
  \end{pmatrix}
  \wedge
  \begin{pmatrix}
    \hat{\va}\\
    \hat{\vb}
  \end{pmatrix},
\end{gather*}
one checks that
\begin{equation*}
  \omega_{\LC} =
  \begin{pmatrix}
    \tilde{\mu} & \varphi \id_{n} + \tilde{\lambda} & \tilde{\vb}
    & \tilde{\va}\\
    -\varphi \id_{n}-\tilde{\lambda} & \tilde{\mu} &-\tilde{\va}
    &\tilde{\vb} \\
    \tallmstrut \tilde{\vb}^{T} & -\tilde{\va}^{T} & 0
    & \varphi\\
    \tallmstrut \tilde{\va}^{T} & \tilde{\vb}^{T} & -\varphi & 0
  \end{pmatrix}
  =
  \begin{pmatrix}
    \varphi \ii_{2n}+\tilde{\omega}_{\LCS} &
    \begin{matrix}
      \ii_{2n}\tilde{\theta}_{S} & \tilde{\theta}_{S}
    \end{matrix}\\
    \begin{matrix}
      \tallmstrut -{\tilde{\theta}_{S}}^{T}\ii_{2n} \\
      \tallmstrut \tilde{\theta}_{S}^{T}
    \end{matrix}
    & \varphi \ii_{2}
  \end{pmatrix}
  .
\end{equation*}

The conditions that \( (S,g_{S},\omega_{S}) \) be \emph{projective
special Kähler} are that its cone \( (C,g_{C},\omega_{C},X) \) is
special Kähler with \( X \) a conic symmetry.
More precisely this means that \( C \) admits a torsion-free flat
symplectic connection~\( \nabla \) with
\( (\nabla_{A}J)B = (\nabla_{B}J)A \), for all \( A,B \), and such
that the symmetry~\( X \) satisfies \( \nabla X = -J \).
Note that our construction of \( C \) already ensures that
\( (g_{C},\omega_{C},J) \) is pseudo-Kähler, that \( X \) is non-null
and that \( \nabla^{\LC} X = -J \).

The conditions on \( \nabla \) were carefully analysed
in~\cite{Macia-S:c-map}, summarised there in the proof of
Proposition~6.3, giving the following: writing the connection one-form
for \( \nabla \) in the coframe~\( \theta_{C} \) as
\( \omega_{\nabla} \), the conic special Kähler conditions on the
pseudo-Kähler manifold~\( C \) are equivalent to the existence of a
matrix-valued one-form~\( \eta = \omega_{\nabla} - \omega_{\LC} \)
such that
\begin{enumerate}
\item\label{item:flat}
  \( \Omega_{\nabla} = d\omega_{\nabla} +
  \omega_{\nabla}\wedge\omega_{\nabla} = 0 \) (flat),
\item\label{item:torsion-free} \( \eta\wedge\theta = 0 \)
  (torsion-free),
\item\label{item:special-symplectic-i}
  \( \ii_{2n+2}\eta = -\eta\ii_{2n+2} \) (special symplectic),
\item\label{item:special-symplectic-G} \( \eta^{T}\GG = -\GG\eta \)
  (special symplectic),
\item\label{item:X-N-LC} \( X \hook \eta = 0 \) (conic),
\item\label{item:X-conic} \( JX \hook \eta = 0 \) (conic).
\end{enumerate}

\begin{lemma}
  \label{lem:eta}
  For \( S \) of real dimension \( 2n \), the difference
  \( \eta = \omega_{\nabla} - \omega_{\LC} \)of the special and
  Levi-Civita connections on~\( C \) is given by a matrix of one forms
  with the following structure:
  \begin{equation*}
    \begin{split}
      \eta =
      \begin{pmatrix}
        \vu&\vv&0&0\\
        \vv&-\vu&0&0\\
        0&0&0&0\\
        0&0&0&0
      \end{pmatrix}
           &\in C^{\infty}(M) \otimes \pi^{*}\Omega^{1}(S,M_{2n+2}(\bR))\\
           &\subset \Omega^{1}(C,M_{2n+2}(\bR)),
    \end{split}
  \end{equation*}
  where \( \vu \) and~\( \vv \) take values in symmetric
  \( n\times n \) matrices and satisfy
  \begin{equation}
    \label{eq:torsion-free}
    \vu \wedge \hat{\va} + \vv \wedge \hat{\vb}
    = 0
    = \vv \wedge \hat{\va} - \vu \wedge \hat{\vb}.
  \end{equation}
\end{lemma}

\begin{proof}
  Let us write \( \eta \) in block form:
  \begin{equation*}
    \eta =
    \begin{pmatrix}
      \vu & \vv & \vc\\
      \vd & \ve & \vf\\
      \vg & \vh & \vk
    \end{pmatrix}
    ,
  \end{equation*}
  where \( \vu,\vv,\vd,\ve \in M_{n}(\bR) \),
  \( \vc,\vf,\vg^{T},\vh^{T} \in M_{n,2}(\bR) \), and
  \( \vk\in M_{2}(\bR) \).

  Then, first special symplectic
  condition~\ref{item:special-symplectic-i} implies:
  \begin{equation*}
    \ve = -\vu, \quad
    \vd = \vv, \quad
    \vf = - \vc \ii_{2}, \quad
    \vh = \ii_{2} \vg, \quad
    \ii_{2}\vk = -\vk \ii_{2}.
  \end{equation*}

  Writing \( \vc \in M_{n,2} \) as two columns \( \vc =
  \begin{pmatrix}
    \vp&\vq
  \end{pmatrix}
  \), \( \vp,\vq \in M_{n,1}(\bR) \), we have
  \begin{equation*}
    \vf = - \vc \ii _{2} =
    \begin{pmatrix}
      \vq&-\vp
    \end{pmatrix}
    .
  \end{equation*}
  Analogously, writing \( \vg \in M_{2,n} \) as two rows, named
  \( \vl, \vm \in M_{1,n}(\bR) \), we have
  \begin{equation*}
    \vh = \ii _{2} \vg =
    \begin{pmatrix}
      \vm\\
      -\vl
    \end{pmatrix}
    .
  \end{equation*}
  Finally, \( \ii_{2}\vk = -\vk\ii_{2} \) implies that \( \vk \) is a
  symmetric traceless matrix, thus leading to
  \begin{equation*}
    \eta =
    \begin{pmatrix}
      \vu & \vv & \vp & \vq \\
      \vv & -\vu & \vq & -\vp\\
      \vl & \vm & x & y\\
      \vm & -\vl & y & -x
    \end{pmatrix}
    ,
  \end{equation*}
  for some scalar-valued one-forms \( x \),~\( y \).

  The second symmetry to be exploited
  is~\ref{item:special-symplectic-G}, which gives
  \begin{equation*}
    \vu^{T} = \vu,\quad
    \vv^{T} = \vv,\quad
    \vm = -\vq^{T},\quad
    \vl = -\vp^{T}.
  \end{equation*}
  Thus, we obtain
  \begin{equation*}
    \eta =
    \begin{pmatrix}
      \vu & \vv & \vp & \vq \\
      \vv & -\vu & \vq & -\vp\\
      -\vp^{T} & -\vq^{T} & x & y \\
      -\vq^{T} & \vp^{T} & y & -x
    \end{pmatrix}
    , \qquad
    \vu^{T} = \vu,\quad
    \vv^{T} = \vv.
  \end{equation*}

  Next, the torsion-free condition~\ref{item:torsion-free} gives two
  vector, and two scalar equations
  \begin{align}
    \label{eq:tf1}
    \vu \wedge \hat{\va}+\vv \wedge \hat{\vb} + \vp \wedge\hat{\varphi} +
    \vq \wedge \hat{\psi} &= 0,\\
    \label{eq:tf2}
    \vv \wedge \hat{\va}-\vu \wedge \hat{\vb} + \vq \wedge\hat{\varphi} -
    \vp\wedge \hat{\psi} &= 0,\\
    \label{eq:tf3}
    -\vp^{T} \wedge \hat{\va}-\vq^{T} \wedge \hat{\vb }+ x\wedge\hat{\varphi}
    + y\wedge\hat{\psi} &= 0,\\
    \label{eq:tf4}
    -\vq^{T}\wedge \hat{\va}+\vp^{T}\wedge\hat{\vb} + y \wedge \hat{\varphi}
    -x\wedge\hat{\psi} &= 0.
  \end{align}
  However, the conic conditions~\ref{item:X-N-LC}
  and~\ref{item:X-conic} imply that each entry of \( \eta \) pointwise
  lies in the space of the components of \( \hat{a} \) and
  \( \hat{b} \).
  Thus~\eqref{eq:tf1} and~\eqref{eq:tf2} imply that \( p = 0 = q \),
  and then~\eqref{eq:tf1} and~\eqref{eq:tf2} show that
  \( x = 0 = y \).
  We thus have the claimed equation~\eqref{eq:torsion-free}.
\end{proof}

\begin{remark}
  \label{rem:total-symmetry}
  Resorting to index notation, and expanding in the one-forms
  \( \{\hat{a}^{i},\hat{b}^{i},\hat{\varphi},\hat{\psi}\} \), we can
  write
  \begin{equation*}
    \vu = (u^{i}_{j}),\quad\text{with}\quad
    u^{i}_{j} = u^{i}_{ajk}\hat{a}^{k} + u^{i}_{bj\ell}
    \hat{b}^{\ell},\quad
    \text{for}\ 1\leqslant i,j,k,\ell\leqslant n,
  \end{equation*}
  and similar expressions for \( \vv \).
  Considering the \( \hat{a}^{i}\wedge \hat{b}^{j} \)-terms in the
  torsion-free condition~\eqref{eq:torsion-free} then gives
  \begin{equation}
    \label{eq:uv-comp}
    u^{i}_{ajk} = -v^{i}_{bjk},\quad u^{i}_{bjk} = v^{i}_{ajk},
  \end{equation}
  hence
  \begin{equation*}
    u^{i}_{j}  = -v^{i}_{bjk} \hat{a}^{k} + v^{i}_{ajk} \hat{b}^{k}.
  \end{equation*}
  So using the complex structure, we have
  \begin{equation}
    \label{eq:u-Jv}
    \vu = J\vv.
  \end{equation}

  Furthermore, the symmetry of matrices \( \vu \) and~\( \vv \)
  together with the relations~\eqref{eq:uv-comp} gives that the
  coefficients \( v^{i}_{ajk} \) and~\( v^{i}_{bjk} \) are totally
  symmetric under permutation of all indices:
  \begin{gather}
    \label{eq:va-sym}
    v^{i}_{ajk} = v^{j}_{aki} = v^{k}_{aij} = v^{i}_{akj} =
    v^{j}_{aki} = v^{k}_{aji},
    \\
    \label{eq:vb-sym}
    v^{i}_{bjk} = v^{j}_{bki} = v^{k}_{bij} = v^{i}_{bkj} =
    v^{j}_{bki} = v^{k}_{bji},
  \end{gather}
  and therefore determine a symmetric three-tensor on~\( C \).
  This is the standard holomorphic three-tensor associated to special
  complex geometry, cf.~\cite{Freed:special,Lledo-MvPV:special}.
\end{remark}

The one-form for the special connection
\( \omega_{\nabla} = \omega_{\LC}+\eta \) is now
\begin{equation*}
  \omega_{\nabla} =
  \begin{pmatrix}
    \tilde{\mu} + \vu & \varphi \id_{n} + \tilde{\lambda} + \vv &
    \tilde{\vb}
    & \tilde{\va}\\
    -\varphi \id_{n}-\tilde{\lambda} +\vv & \tilde{\mu}-\vu
    &-\tilde{\va}
    & \tilde{\vb}\\
    \tallmstrut \tilde{\vb}^{T} & -\tilde{\va}^{T} & 0
    & \varphi\\
    \tilde{\va}^{T} & \tilde{\vb}^{T} & -\varphi & 0
  \end{pmatrix}
  .
\end{equation*}
Using the torsion-free equations~\eqref{eq:torsion-free}, one finds
that the curvature of the special connection is given by
\begin{equation*}
  \Omega_{\nabla} =
  \begin{pmatrix}
    T+U & V+W & 0 & 0 \\
    V-W & T-U & 0 & 0\\
    0 & 0 & 0 & 0\\
    0 & 0 & 0 & 0
  \end{pmatrix}
  ,
\end{equation*}
where
\begin{align}
  \label{eq:T}
  T &= \tilde{\mu} + \tilde{\va} \wedge
      \tilde{\va}^{T} + \tilde{\vb} \wedge \tilde{\vb}^{T} + \vu\wedge\vu +
      \vv\wedge\vv,\\
  \label{eq:U}
  U &= d\vu + \tilde{\mu}\wedge \vu + \vu\wedge \tilde{\mu}
      + \tilde{\lambda}\wedge \vv - \vv\wedge \tilde{\lambda} +
      2\varphi\wedge \vv,
  \\
  \label{eq:V}
  V &= d\vv + \tilde{\mu}\wedge\vv + \vv\wedge\tilde{\mu}
      - \tilde{\lambda}\wedge\vu + \vu\wedge\tilde{\lambda}
      - 2\varphi\wedge\vu,\\
  \label{eq:W}
  W &= \tilde{\lambda} + \tilde{\va}\wedge\tilde{\vb}^{T} -
      \tilde{\vb}\wedge\tilde{\va}^{T} + 2\pi^{*}\omega_{S}\id_{n}
      + \vu\wedge\vv - \vv\wedge\vu.
\end{align}
As the special conditions requires \( \Omega_{\nabla} = 0\), the next
result summarizes the situation.

\begin{proposition}
  \( S \) is projective special Kähler if and only if on the
  cone~\( C \) there is a one-form~\( \eta \) as in Lemma~\ref{lem:eta} so
  that the expressions~\eqref{eq:T}--\eqref{eq:W} satisfy
  \( T = U = V = W = 0 \).
\end{proposition}

\begin{remark}
  \label{rem:flat}
  The special connection reduces to the Levi-Civita connection of the
  cone exactly when \( \vu = \vv = 0 \).
  In this situation, \eqref{eq:torsion-free} is satisfied and
  \( U = 0 = V \).
  What remains are the equations \( T = 0 = W \), now determine the
  Kähler curvature of~\( S \):
  \begin{equation}
    \label{eq:curv-CH}
    \Mu = - \va\wedge\va^{T} - \vb\wedge\vb^{T} \eqand
    \Lambda = - \va\wedge\vb^{T} + \vb\wedge\va^{T}
    - 2\omega_{S}\id_{n}.
  \end{equation}
  Thus
  \begin{equation*}
    \Omega_{S} = - \{ \theta\wedge\theta^{T} + (J \theta)\wedge
    (J\theta)^{T} + 2\omega_{S} J \},
  \end{equation*}
  which is the curvature tensor of complex hyperbolic
  space~\( \CH(n) \) with holomorphic sectional curvature~\( -1 \).
  In this case, we have
  \begin{equation*}
    \mu =
    \begin{pmatrix}
      0 & \va_{R}^{T}\\
      -\va_{R} & 0_{n-1}
    \end{pmatrix}
    \eqand
    \lambda =
    \begin{pmatrix}
      2b_{1} & \vb_{R}^{T}\\
      \vb_{R}& b_{1}\id_{n-1}
    \end{pmatrix}
    ,
  \end{equation*}
  where \( \va_{R} = (a^{2},\dots,a^{n}) \), etc.
\end{remark}

\begin{example}
  \label{ex:CH1}
  The case the complex hyperbolic line, i.e.\ \( S = \CH(1) \), with
  arbitrary (negative) holomorphic sectional curvature, was discussed
  in detail in~\cite{Macia-S:c-map}.
  One has \( \mu = 0 \) and \( \lambda = -cb \) for some
  \( c\in\bR \), where the holomorphic sectional curvature is
  \( -c^{2}/4 \).
  It was shown that there are only solutions to the projective special
  Kähler equations in the cases when the holomorphic sectional
  curvature is \( -1 \) or \( -1/3 \).
\end{example}

\section{Integrability equations}
\label{sec:integrability}

To understand the special geometry better, let us consider the
integrability conditions related to the torsion-free condition and the
vanishing of \( T \), \( U \), \( V \), and \( W \).
Differentiating \( U = 0 \), and substituting for \( du \) and
\( dv \) from \( U = 0 = V \), we get
\begin{equation}
  \label{eq:dU}
  0 = \tilde{\Mu}\wedge\vu - \vu\wedge\tilde{\Mu}
  + \tilde{\Lambda}\wedge\vv + \vv\wedge\tilde{\Lambda}
  + 4\pi^{*}\omega_{S}\wedge\vv
\end{equation}
after substituting the expressions for \( d\vu \) and~\( d\vv \) from
\( U = 0 = V \).  Similarly, differentiating \( V = 0 \) gives
\begin{equation}
  \label{eq:dV}
  0 = \tilde{\Mu}\wedge\vv - \vv\wedge\tilde{\Mu}
  - \tilde{\Lambda}\wedge\vu - \vu\wedge\tilde{\Lambda}
  - 4\pi^{*}\omega_{S}\wedge\vu.
\end{equation}
It is tempting to substitute for \( \tilde{\Mu} \) and
\( \tilde{\Lambda} \) using \( T = 0 = W \), but after applying the
torsion-free condition~\eqref{eq:torsion-free} this yields no
information.
Thus these equations are consequences of the torsion-free relation and
the vanishing of \( T \) and~\( W \).
However, these equations can provide useful constraints on \( u \) and
\( v \) as we will see below.
Similar considerations show that the system of equations
\( T=0=U=V=W \), \eqref{eq:torsion-free}, together with the lifts of
\eqref{eq:structural}, \eqref{eq:Mu-Lambda}, and the differential
of~\eqref{eq:Mu-Lambda},
\begin{gather*}
  d\Mu = \Mu \wedge \mu - \mu \wedge \Mu - \Lambda \wedge \lambda +
  \lambda \wedge \Lambda,\\
  d\Lambda = \Mu \wedge \lambda - \mu \wedge \Lambda + \Lambda \wedge
  \mu - \lambda \wedge \Mu,
\end{gather*}
in the variables
\( \tilde{\mu},\tilde{\lambda},\tilde{\Mu},\tilde{\Lambda},\vu,\vv \)
is closed under exterior derivatives.

\section{Flat factors}
\label{sec:flat-factors}

Following~\cite{Lichnerowicz-M:symplectic-announce}, it is reasonable
to study the situation when the Kähler group~\( S \) is a product
\( \widetilde{S} = S^{0} \times R \) with \( S^{0} \) flat and Kähler.
Note that statements in~\cite{Lichnerowicz-M:Kaehlerian} modify those
of the previous reference, and the correctness of those results is not
clear.

If \( S = S^{0} \times R \) with \( S^{0} \) flat and Kähler, then we
can split \( \va = (a_{0}, a_{R})^{T} \) etc.\ and write
\begin{equation*}
  \vu =
  \begin{pmatrix}
    \vu_{0} & \vu_{+} \\
    \vu_{-} & \vu_{R}
  \end{pmatrix}
  \qquad
  \mu =
  \begin{pmatrix}
    \mu_{0} & 0 \\
    0 & \mu_{R}
  \end{pmatrix}
\end{equation*}
and so on, with \( \vu_{-} = (\vu_{+})^{T} \).  It follows that
\begin{equation*}
  \Mu =
  \begin{pmatrix}
    0 & 0 \\
    0 & \Mu_{R}
  \end{pmatrix}
  \quad
  \Lambda =
  \begin{pmatrix}
    0 & 0 \\
    0 & \Lambda_{R}
  \end{pmatrix}.
\end{equation*}
Now, suppressing wedge signs, \eqref{eq:dU}~reads
\begin{equation}
  \label{eq:dU-flat}
  0 =
  \begin{pmatrix}
    4\pi^{*}\omega_{S}\vv_{0} & - \vu_{+}\tilde{\Mu}_{R} +
    \vv_{+}\tilde{\Lambda}_{R} + 4\pi^{*}\omega_{S}\vv_{+}\\
    \tilde{\Mu}_{R}\vu_{-}+\tilde{\Lambda}_{R}\vv_{-}+4\pi^{*}\omega_{S}\vv_{-}
    &(dU)_{R}
  \end{pmatrix},
\end{equation}
The \( (0,0) \)-component of this equation gives
\begin{equation*}
  4\pi^{*}\omega_{S}\wedge \vv_{0} = 0,
\end{equation*}
which for \( \dim_{\bR}S > 2 \), implies \( \vv_{0} = 0 \).
Similarly, from~\eqref{eq:dV} we get \( \vu_{0} = 0 \).
The symmetries from the torsion-free conditions, \eqref{eq:u-Jv},
\eqref{eq:va-sym} and~\eqref{eq:vb-sym}, now imply that \( \vu_{+} \)
and \( \vv_{+} \) have no \( \tilde{a}_{0} \)- or
\( \tilde{b}_{0} \)-components.
Now the \( + \)-component of~\eqref{eq:dU-flat} only contains
\( \tilde{a}_{0} \) and \( \tilde{b}_{0} \) terms
in~\( \pi^{*}\omega_{S} \), so
\( 4\pi^{*}\omega_{0}\wedge\vv_{+} = 0 \), giving \( \vv_{+} = 0 \),
and \( \vv_{-} = \vv_{+}^{T} = 0 \); similarly,
\( \vu_{+} = 0 = \vu_{-} \).

The \( (0,0) \)-component of \( W = 0 \), is
\( \tilde{\va}_{0} \wedge \tilde{\vb}_{0} - \tilde{\vb}_{0} \wedge
\tilde{\va}_{0} + 2\pi^{*}\omega_{S}\id_{m} = 0 \), for
\( \dim_{\bC} S^{0} = m \).
But \( \omega_{S} = \omega_{0} + \omega_{R} \), so if
\( \dim R > 0 \), we get a contradiction.

If \( \dim R = 0 \), then we have \( \dim_{\bR}S = 2 \) and
\( S = S^{0} \).
But then \( S^{0} \) is Abelian, \( \lambda = 0 = \mu \) and the
vanishing of \( U \) and~\( V \)~\eqref{eq:U} and~\eqref{eq:V}, imply
\( \vu = 0 = \vv \).
Finally \( W = 0 \), gives
\( 4\tilde{\va}_{0}\wedge \tilde{\vb}_{0} = 0 \), a contradiction.

Thus we have proved:

\begin{proposition}
  \label{prop:no-flat}
  Suppose \( S \) is a Lie group with a left-invariant Kähler
  structure that extends to a (not necessarily left-invariant)
  projective special Kähler structure.
  Then the de Rham decomposition of the universal cover
  \( \widetilde{S} \) has no flat Kähler factor.  \qed
\end{proposition}

\section{Products}
\label{sec:products}

Suppose the universal cover of~\( S \) is a product of Kähler groups,
\( \widetilde{S} = S_{1}\times S_{2} \).
Then splitting \( \va = (\va_{1},\va_{2}) \) etc., we have
\begin{equation*}
  \Mu =
  \begin{pmatrix}
    \Mu_{1}&0 \\
    0&\Mu_{2}
  \end{pmatrix}
  \eqand
  \Lambda =
  \begin{pmatrix}
    \Lambda_{1}&0 \\
    0& \Lambda_{2}
  \end{pmatrix}
  .
\end{equation*}
Writing
\begin{equation*}
  \vu =
  \begin{pmatrix}
    \vu_{1}&\vu_{+}\\
    \vu_{-}&\vu_{2}
  \end{pmatrix}
\end{equation*}
etc., \eqref{eq:dU} has components
\begin{align}
  \label{eq:dU-prod-11}
  0 &= \tilde{\Mu}_{1}\wedge\vu_{1} - \vu_{1}\wedge\tilde{\Mu}_{1}
      + \tilde{\Lambda}_{1}\wedge\vv_{1}
      + \vv_{1}\wedge\tilde{\Lambda}_{1}
      + 4\pi^*\omega_{S}\wedge\vv_{1}, \\
  \label{eq:dU-prod-12}
  0 &= \tilde{\Mu}_{1}\wedge\vu_{+} - \vu_{+}\wedge\tilde{\Mu}_{2}
      + \tilde{\Lambda}_{1}\wedge\vv_{+}
      + \vv_{+}\wedge\tilde{\Lambda}_{2}
      + 4\pi^*\omega_{S}\wedge\vv_{+}, \\
  \label{eq:dU-prod-21}
  0 &= \tilde{\Mu}_{2}\wedge\vu_{-} - \vu_{-}\wedge\tilde{\Mu}_{1}
      + \tilde{\Lambda}_{2}\wedge\vv_{-}
      + \vv_{-}\wedge\tilde{\Lambda}_{1}
      + 4\pi^*\omega_{S}\wedge\vv_{-}, \\
  \label{eq:dU-prod-22}
  0 &= \tilde{\Mu}_{2}\wedge\vu_{2} - \vu_{2}\wedge\tilde{\Mu}_{2}
      + \tilde{\Lambda}_{2}\wedge\vv_{2}
      + \vv_{2}\wedge\tilde{\Lambda}_{2}
      + 4\pi^*\omega_{S}\wedge\vv_{2}.
\end{align}
Let us use \( V^{a,b} \) to denote the bundle
\( \pi^{*}(\Lambda^{a}T^{*}S_{1} \wedge \Lambda^{b}T^{*}S_{2}) \).
The terms in~\eqref{eq:dU-prod-11} in \( V^{1,2} + V^{0,3} \) are
\( 4\pi^{*}\omega_{S_{2}}\wedge \vv_{1} \), so if
\( \dim_{\bR} S_{2} > 2 \), we have \( \vv_{1} = 0 \).
Similarly, using~\eqref{eq:dV}, we get \( \vu_{1} = 0 \).
Now, under this condition, the total symmetry~\eqref{eq:va-sym}
and~\eqref{eq:vb-sym} implies \( \vu_{+} \) and \( \vv_{+} \) consist
of one-forms in~\( V^{0,1} \).

Similarly, if we also have \( \dim_{\bR} S_{1} > 2 \), then
\( \vv_{2} = 0 = \vu_{2} \), and \( \vu_{\pm} = 0 = \vv_{\pm} \),
i.e.\ \( \vu = 0 = \vv \).
But then \( S = \CH(n) \) which is not a product.
Thus a product structure has at least one factor of real
dimension~\( 2 \).

\begin{proposition}
  \label{prop:three}
  Suppose the universal cover \( \widetilde{S} \) is product of three
  or more Kähler factors.
  Then there are exactly three factors and
  \( \widetilde{S} = \CH(1)\times \CH(1)\times \CH(1) \), each with
  holomorphic section curvature~\( -1 \).
\end{proposition}

\begin{proof}
  Write \( \widetilde{S} = S_{A} \times S_{B} \times S_{C} \), with
  each factor of non-zero dimension.
  Then \( \dim_{\bR}(S_{A}\times S_{B}) \) is strictly greater
  than~\( 2 \), so \( \dim_{\bR} S_{C} = 2 \).
  Grouping in different ways yields
  \( \dim_{\bR}S_{A} = 2 = \dim_{\bR}S_{B} \).
  By Proposition~\ref{prop:no-flat}, these factors are not flat, so are each
  isomorphic to \( \CH(1) \).
  In particular, \( M_{i} = 0 \) and
  \( \Lambda_{i} = r_{i}\omega_{S_{i}} \), for \( i=A,B,C \), with
  \( r_{i} \in\bR\setminus\{0\} \).

  Writing
  \begin{equation*}
    \vu =
    \begin{pmatrix}
      u_{A}&u_{AB}&u_{AC}\\
      u_{BA}&u_{B}&u_{BC}\\
      u_{CA}&u_{CB}&u_{C}
    \end{pmatrix}
    ,\quad\text{etc.}
  \end{equation*}
  equation~\eqref{eq:dU} gives
  \( \pi^{*}(2r_{A}\omega_{S_{A}}+4\omega_{S})v_{A} = 0 \), implying
  \( v_{A} = 0 \), and so all diagonal entries in \( u \) and \( v \)
  are zero.
  For the off-diagonal terms, we have
  \( \pi^{*}((r_{A}+4)\omega_{S_{A}} +
  (r_{B}+4)\omega_{S_{B}}+\omega_{S_{C}})v_{AB} = 0 \).
  But \( u \) and \( v \) are not identically zero, so
  \( r_{i} = -4 \) for all \( i \), and
  \( v_{AB} \in \pi^{*}T^{*}S_{C} \), etc.
  Total symmetry of the \( u \) and \( v \) implies
  \( u_{AB} = u'\tilde{a}_{C} + u''\tilde{b}_{C} \),
  \( v_{AB} = u''\tilde{a}_{C} - u'\tilde{b}_{C} \), and cyclically,
  for some smooth functions \( u',u'' \) on the cone.
  The equations \( T = 0 = W \) only impose the constraint
  \( (u')^{2} + (u'')^{2} = 1 \), so \( u' = \cos(s) \),
  \( u''= \sin(s) \) for some smooth local function~\( s \).
  Equations~\eqref{eq:U} and~\eqref{eq:V}, then give
  \( ds = 2(\varphi + \tilde{b}_{A} + \tilde{b}_{B} + \tilde{b}_{C})
  \), which is an exact form, so \( s \)~is globally defined.
\end{proof}

\section{Rotations and intrinsic equations}
\label{sec:rotations}

Let \( S \) be as in Proposition~\ref{prop:no-flat}.
Consider its cone~\( C \).
The conic vector field \( X \) satisfies \( X\hook\varphi = 1 \),
\( L_{X}\varphi = 0 \), and \( X\hook\vu = 0 = X\hook\vv \).
Thus the vanishing of \( U \) and~\( V \) gives
\begin{gather*}
  L_{X}\vu = X\hook d\vu = -2\vv,\quad
  L_{X}\vv = X\hook d\vv = 2\vu,\\
  L_{JX}\vu = JX\hook d\vu = 0,\quad L_{JX}\vv = JX\hook d\vv = 0.
\end{gather*}
In particular, the matrices of one-forms \( \vu \) and~\( \vv \) are
not invariant under the conic symmetry and they do not descend
to~\( S \) even though they vanish on \( X \) and~\( JX \).

Suppose the curvature of the circle bundle \( C_{0} \to S \) is exact:
\( d\varphi = 2\pi^{*}\omega_{S} = 2\pi^{*}d\kappa \) for some
\( \kappa \in \Omega^{1}(S) \).
Then \( \varphi' = \varphi - 2\pi^{*}\kappa \) is a flat connection on
\( C_{0} \), so the pull-back \( C'_{0} \to \widetilde{S} \) to the
universal cover is a trivial circle bundle.
Let \( C' = \bR_{>0}\times C'_{0} \) be the special Kähler cone of
\( \widetilde{S} \) and choose a trivialisation of
\( C'_{0} \cong S^{1} \times \widetilde{S} \), writing points of
\( S^{1} \) as \( e^{\ii \tau} \), then \( X = \partial_{\tau} \) and
\( \varphi = d\tau + 2\pi^{*}\kappa \).

Consider new matrices of one-forms \( \vP \), \( \vQ \) obtained by
rotating the pair \( \vu \), \( \vv \) through some angle \( z \),
that is
\begin{equation*}
  \vP = \vu\cos z + \vv\sin z,\quad
  \vQ = -\vu\sin z + \vv\cos z.
\end{equation*}
Then, we find
\begin{align*}
  L_{X}\vP
  &= L_{X}\vu \cos z - \vu \sin z\, (X\hook dz) + L_{X}\vv \sin z
    + \vv \cos z \, (X\hook dz)\\
  &= (2 - (Xz)) \vu \sin z - (2 - (Xz)) \vv \cos z,
  \\
  L_{JX}\vP &= -((JX)z) \vu \sin z + ((JX)z) \vv \cos z,
\end{align*}
with a similar expressions for Lie derivatives of~\( \vQ \).
Putting \( z = 2\tau \) we get \( L_{X}\vP = L_{X}\vQ = 0 \) and hence
that \( \vP \) and \( \vQ \) are basic.
Then \( \vP = \pi^{*}\vp \) and \( \vQ = \pi^{*}\vq \), for some
one-forms on~\( S \) with values in \( n\times n \) symmetric
matrices.
Furthermore, this essentially the only choice: if \( x\vu+y\vv \) is
basic and nowhere zero, with \( x \) and~\( y \) smooth functions
on~\( C' \), then there is a trivialisation of \( C_{0}' \) such that
\( x = r\cos(2\tau+s) \) and \( y = r\sin(2\tau+s) \), with \( r>0 \)
the pull-back of a smooth function on~\( \widetilde{S} \) and with
\( s \)~a real constant.

\begin{definition}
  Let \( S \) be a Kähler group with \( \omega_{S} = d\kappa \) for
  some left-invariant form~\( \kappa \).
  A compatible projective special Kähler structure on \( S \) is
  \emph{homogeneous} if \( \vP \) and \( \vQ \) above are pull-backs
  of left-invariant matrix-valued one-forms \( \vp \) and \( \vq \)
  on~\( S \).
\end{definition}

Note that all group manifold examples of projective special Kähler
structures in the literature satisfy this definition.

We may now rewrite equations~\eqref{eq:T}--\eqref{eq:W} in terms
of~\( \vP,\vQ \) instead of \( \vu, \vv \).
We will see that the resulting equations are determined by
\( \vp, \vq \) on~\( S \).
First, notice that since \( \vu = J\vv \), we have \( \vP = J\vQ \).
Inverting the equations defining \( \vP \), \( \vQ \), gives
\begin{equation*}
  \vu = \vP\cos z - \vQ\sin z,\qquad
  \vv = \vP\sin z + \vQ\cos z.
\end{equation*}
Substituting this in the torsion-free
condition~\eqref{eq:torsion-free} we see the corresponding equation in
\( \vP,\vQ \) and get that it is equivalent to
\begin{equation}
  \label{eq:PQ-torsion-free}
  \vp \wedge\va + \vq \wedge \vb = 0
  = \vp \wedge\vb - \vq \wedge \va.
\end{equation}
Next, although \( \vu \) and \( \vv \) depend on the
angle function~\( z \), we have
\begin{equation*}
  \vu\wedge\vu + \vv\wedge\vv = \vP\wedge\vP + \vQ\wedge\vQ
  \eqand
  \vu\wedge\vv - \vv\wedge\vu = \vP\wedge\vQ - \vQ\wedge\vP.
\end{equation*}
Thus both sides of these equations are \( X \)-invariant.
We now see that the vanishing \( T,W \) is equivalent to
\begin{gather}
  \label{eq:T-PQ}
  \Mu + \vp\wedge\vp + \vq\wedge\vq = \Mu_{\CH},\\
  \label{eq:W-PQ}
  \Lambda + \vp\wedge\vq - \vq\wedge\vp = \Lambda_{\CH},
\end{gather}
on~\( S \), where \( \Mu_{\CH} \) and \( \Lambda_{\CH} \) are the
blocks of the curvature of \( \CH(n) \) given in \eqref{eq:curv-CH}.
On the other hand the expressions for \( U \), \( V \) depend on the
angle function~\( z \).  More precisely,
\begin{align*}
  U
  &= \bigl(d\vP + (\tilde{\mu}\wedge\vP + \vP\wedge\tilde{\mu})
    + (\lambda\wedge\vQ - \vQ\wedge\lambda) - 4\kappa\wedge\vQ \bigr)
    \cos z
    \eqbreak
    - \bigl(d\vQ + (\tilde{\mu}\wedge\vQ + \vQ\wedge\tilde{\mu})
    - (\tilde{\lambda}\wedge\vP - \vP\wedge\tilde{\lambda})
    + 4\kappa\wedge\vP \bigr)\sin z ,\\
  V
  &= \bigl(d\vP + (\tilde{\mu}\wedge\vP + \vP\wedge\tilde{\mu})
    + (\lambda\wedge\vQ - \vQ\wedge\lambda) - 4\kappa\wedge\vQ \bigr)
    \sin z
    \eqbreak
    + \bigl(d\vQ + (\tilde{\mu}\wedge\vQ +\vQ\wedge\tilde{\mu})
    - (\tilde{\lambda}\wedge\vP - \vP\wedge\tilde{\lambda})
    + 4\kappa\wedge\vP \bigr) \cos z.
\end{align*}
Therefore, vanishing of \( U \), \( V \) is equivalent to the
relations
\begin{gather}
  \label{eq:dP}
  d\vp + (\mu\wedge\vp + \vp\wedge\mu)
  + (\lambda\wedge\vq - \vq\wedge\lambda) - 4\kappa\wedge\vq = 0,\\
  \label{eq:dQ}
  d\vq + (\mu\wedge\vq + \vq\wedge\mu) - (\lambda\wedge\vp -
  \vp\wedge\lambda) + 4\kappa\wedge\vp = 0
\end{gather}
on~\( S \).

\begin{proposition}
  A simply-connected Kähler group \( S \) of dimension~\( 2n \) with
  exact Kähler form \( \omega_{S} = d\kappa \) admits a compatible
  projective special Kähler structure if and only if there are
  matrix-valued one-forms \( \vp,\vq \in \Omega^{1}(S,M_{n}(\bR))\)
  on~\( S \) satisfying \( \vp^{T} = \vp \), \( \vq^{T} = \vq \),
  \( \vp = J\vq \), the torsion-free
  condition~\eqref{eq:PQ-torsion-free}, the equations~\eqref{eq:T-PQ},
  \eqref{eq:W-PQ}, \eqref{eq:dP}, and~\eqref{eq:dQ}.

  This structure is homogeneous projective special Kähler if and only
  if \( \vp \) and \( \vq \) can be chosen left-invariant.  \qed
\end{proposition}

\begin{remark}
  \label{rem:rotate}
  The function \( z \), or correspondingly the parameter~\( \tau \),
  is only defined up to a global additive constant.
  Therefore any choice of \( (\vp,\vq) \) can be replaced by a rotated
  version
  \( R_{s}(\vp,\vq) = (\vp\cos s + \vq\sin s,-\vp\sin s+\vq\cos s) \)
  for any constant~\( s \).
  The projective special conditions \eqref{eq:PQ-torsion-free},
  \eqref{eq:T-PQ}, \eqref{eq:W-PQ}, \eqref{eq:dP} and~\eqref{eq:dQ},
  are equivalent to the same system for \( R_{s}(\vp,\vq) \).
\end{remark}

\begin{remark}
  \label{rem:total-symmetry-pq}
  The discussion of Remark~\ref{rem:total-symmetry}, implies that the
  torsion-free condition is equivalent to total symmetry conditions of
  the form~\eqref{eq:va-sym}--\eqref{eq:vb-sym} for the
  \( \va,\vb \)-coefficients of \( p \) and~\( q \).
\end{remark}

\begin{remark}
  The integrability equations of~\S\ref{sec:integrability} are
  equivalent to the pair of equations
  \begin{align}
    \label{eq:int-pq1}
    \Mu\wedge \vp - \vp\wedge\Mu + \Lambda\wedge\vq + \vq\wedge\Lambda
    + 4\omega_{S}\wedge\vq &= 0,\\
    \label{eq:int-pq2}
    \Mu\wedge \vq - \vq\wedge\Mu - \Lambda\wedge\vp - \vp\wedge\Lambda
    - 4\omega_{S}\wedge\vp &= 0
  \end{align}
  on~\( S \).
  These have the advantage of not involving the one-form~\( \kappa \).
\end{remark}

\section{A four-dimensional example}
\label{sec:4d}

In this section we consider the Kähler products
\( \widetilde S = \CH(1)\times \CH(1) \).
We choose our \( \va,\vb \) compatible with the product, so that
\( da_{1} = 0 = da_{2} \) and \( db_{i} = c_{i}a_{i}\wedge b_{i} \),
\( i=1,2 \), with \( c_{i} \ne 0 \).
Replacing \( (a_{i},b_{i}) \) by \( (-a_{i},-b_{i}) \) if necessary,
we may ensure that \( c_{i} > 0 \).
Then
\( \omega_{S} = a_{1}\wedge b_{1} + a_{2}\wedge b_{2} =
d(b_{1}/c_{1}+b_{2}/c_{2}) \) is exact with left-invariant primitive,
so we may apply the equations of~\S\ref{sec:rotations} with the
splittings of~\S\ref{sec:products}.
We have \( \Mu_i=0 \), \( \Lambda_i = r_i \omega_{S_i} \),
\( r_i = - c_{i}^{2} < 0 \).
Moreover, \( \vp,\vq \) are \( M_2(\mathbb R) \)-valued one-forms.
Hence, \( p_1,p_{\pm},p_2,q_1,q_{\pm},q_2 \) are scalar one-forms.
The symmetries \( p^T=p \), \( q^{T}=q \), and \( p=Jq \) imply that
\( p \) and \( q \) are determined by, e.g., \( q_{1},q_{+},q_{2} \).

Comparing with the discussion in~\S\ref{sec:products}, consider the
bundles
\( W^{a,b} = \Lambda^{a}T^{*}S_{1} \wedge \Lambda^{b}T^{*}S_{2} \).
The integrability equation \eqref{eq:int-pq1} gives three equations
\begin{align}
  \label{eq:2CH1-1}
  0 &= (2r_{1}+4)\omega_{S_{1}}\wedge q_{1} + 4\omega_{S_{2}}\wedge
      q_{1},\\
  \label{eq:2CH1-plus}
  0 &= (r_{1}+4)\omega_{S_{1}}\wedge q_{+} +
      (r_{2}+4)\omega_{S_{2}}\wedge q_{+},\\
  \label{eq:2CH1-2}
  0 &= 4\omega_{S_{1}}\wedge q_{2} + (2r_{2}+4)\omega_{S_{2}}\wedge
      q_{2}.
\end{align}
The last term of \eqref{eq:2CH1-1} implies \( q_{1} \in W^{0,1} \);
similarly, the first term of \eqref{eq:2CH1-2} gives
\( q_{2} \in W^{1,0} \).
By the total symmetry of Remark~\ref{rem:total-symmetry-pq}, we see that
\( q_{1} = 0 = q_{2} \) implies \( \vq = 0 \); then
\( \vp = J\vq = 0 \) too, which is a contradiction
with~\eqref{eq:T-PQ}.
Thus, without loss of generality, we may assume that \( q_{1} \) is
non-zero in \( W^{0,1} \).
The first part of \eqref{eq:2CH1-1}, then implies \( r_{1} = -2 \).
Now considering equation \eqref{eq:2CH1-plus}, the first coefficient
is non-zero, so \( q_{+} \in W^{1,0} \); then the second term forces
\( r_{2} = -4 \); so the holomorphic sectional curvatures of the two
factors are \( -1/2 \) and \( -1 \).
Equation~\eqref{eq:2CH1-2} now implies \( q_{2} = 0 \).

Let us now consider the equations~\eqref{eq:T-PQ}, \eqref{eq:W-PQ},
\eqref{eq:dP} and~\eqref{eq:dQ} Firstly~\eqref{eq:T-PQ} reduces to
\begin{equation*}
  p_{1}\wedge p_{+} + q_{1}\wedge q_{+} = -a_{1}\wedge
  a_{2}-b_{1}\wedge b_{2}.
\end{equation*}
Equation~\eqref{eq:W-PQ} gives three relations
\begin{gather}
  \label{eq:Wpq-diag}
  p_{+}\wedge q_{+} = -a_{1}\wedge b_{1},\quad
  p_{1}\wedge q_{1} = -a_{2}\wedge b_{2},\\
  \label{eq:Wpq-offdiag}
  p_{1}\wedge q_{+}-q_{1}\wedge p_{+} = - a_{1}\wedge
  b_{2}-a_{2}\wedge b_{1}.
\end{gather}
As \( q_{+} \) and \( p_{+} = Jq_{+} \) are span \( T^{*}S_{1} \), we
may use Remark~\ref{rem:rotate} to take \( q_{+} = ra_{1} \) with \( r>0 \).
It follows that \( p_{+} = rb_{1}\) and \eqref{eq:Wpq-diag} implies
\( r=1 \).
Total symmetry now gives \( q_{1} = a_{2} \), \( p_{1} = b_{2} \).
Thus
\begin{equation*}
  \vq =
  \begin{pmatrix}
    a_{2}&a_{1}\\
    a_{1}&0
  \end{pmatrix}
  \quad
  \vp =
  \begin{pmatrix}
    b_{2}&b_{1}\\
    b_{1}&0
  \end{pmatrix}
  .
\end{equation*}
It remains to consider equations \eqref{eq:dP} and~\eqref{eq:dQ}.
We have \( db_{1} = \sqrt{2}a_{1}\wedge b_{1} \),
\( db_{2} = 2a_{2}\wedge b_{2} \), and putting this into \eqref{eq:dP}
and~\eqref{eq:dQ} gives \( 2\kappa = -\sqrt{2}b_{1}-2b_{2} \), so we
indeed have \( d\kappa = \omega_{S} \).

\section{Twists and the quaternionic Kähler metric}

Given a projective special Kähler manifold \( S \), with special
Kähler cone~\( C \) carrying the conic isometry \( X \), the c-map
corresponds to the twist of the cotangent bundle \( H \) of~\( C \).
The lift of the conic isometry leads to a rotating symmetry on \( H \)
which generates the Abelian action needed to twist.

According to the general theory developed in~\cite{Macia-S:c-map}, the
pseudo-hyperKähler metric on \( H = T^{*}C \) is given in terms of the
coframes as
\begin{equation*}
  \begin{split}
    g_{H}
    = \hat{\va}^{T}\hat{\va} + \hat{\vb}^{T}\hat{\vb}
    -\hat{\varphi}^{2}-\hat{\psi}^{2} +\hat{\vA}^{T}\hat{\vA}
    +\hat{\vB}^{T}\hat{\vB} -\hat{\Phi}^{2} - \hat{\Psi}^{2},
  \end{split}
\end{equation*}
where
\( \hat{\gamma} = (\hat{\va},\hat{\vb},\hat{\varphi},\hat{\psi}) \) is
the coframe of~\( C \) given in~\S\ref{sec:sK-cond} and
\( \hat{\delta} = (\hat{\vA},\hat{\vB},\hat{\Phi},\hat{\Psi}) \) is a
corresponding coframe on each \( T_{x}^{*}C \) satisfying
\( d\hat{\delta} = - \hat{\delta}\wedge\omega_{\nabla} \).
The triple of Kähler two-forms is then given by
\begin{align*}
  \omega_{I}
  &= \hat{\va}^{T}\wedge \hat{\vb} -\hat{\varphi}\wedge\hat{\psi}
    + \hat{\vA}^{T}\wedge\hat{\vB} - \hat{\Phi}\wedge\hat{\Psi},\\
  \omega_{J}
  &= \hat{\vA}^{T}\wedge \hat{\va} + \hat{\vB}^{T}\wedge \hat{\vb}
    + \hat{\Phi}\wedge\hat{\varphi} + \hat{\Psi}\wedge\hat{\psi},\\
  \omega_{K}
  &= \hat{\vA}^{T}\wedge \hat{\vb} - \hat{\vB}^{T}\wedge\hat{\va}
    + \hat{\Phi}\wedge \hat{\psi} - \hat{\Psi}\wedge\hat{\varphi}.
\end{align*}
The conic symmetry~\( X \) lifts to a vector field~\( \tilde{X} \)
on~\( H \) preserving \( g_{H} \) and \( \omega_{I} \), but with
\( L_{\tilde{X}}\omega_{J} = \omega_{K} \).

To obtain a quaternionic Kähler metric one first considers the
positive definite metric on~\( H \) given by
\begin{equation*}
  g_{N} = \frac{2}{t^{2}}\bigl(\hat{\va}^{T}\hat{\va}
  + \hat{\vb}^{T}\hat{\vb} + \hat{\varphi}^{2} + \hat{\psi}^{2}
  + \hat{\vA}^{T}\hat{\vA} + \hat{\vB}^{T}\hat{\vB}
  + \hat{\Phi}^{2} + \hat{\Psi}^{2}\bigr),
\end{equation*}
this called an ``elementary deformation'' of \( g_{H} \)
in~\cite{Macia-S:c-map}.
Now one builds a principal circle bundle~\( P\to H \) with principal
generator~\( Y \) and curvature
\begin{equation*}
  F = - \hat{\va}^{T}\wedge\hat{\vb} + \hat{\varphi}\wedge\hat{\psi}
  - \hat{\vA}^{T}\wedge\hat{\vB} + \hat{\Phi}\wedge\hat{\Psi},
\end{equation*}
and constructs the twist \( Q \) of \( H \) as
\( Q = P/(\tilde{X} - t^{2}Y/2) \), the factor \( -t^{2}/2 \) being
the twist function satisfying the condition
\( d(-t^{2}/2) = \tilde{X}\hook F \).
Tensors on~\( H \) invariant under~\( \tilde{X} \) may now be
transferred to~\( Q \), and the exterior differential on~\( Q \)
corresponds to the operation \( d_{Q} \) on~\( H \) given
\begin{equation*}
  d_{Q}\beta = d\beta + \frac{2}{t^{2}}F\wedge (\tilde{X}\hook\beta)
\end{equation*}
for invariant forms~\( \beta \).
The general result of \cite{Macia-S:c-map} is that metric on~\( Q \)
induced by~\( g_{N} \) is quaternionic Kähler.

Since the coframe
\( \gamma = (\tilde{\va},\tilde{\vb},\varphi,\tilde{\psi}) \) is
\( X \)-invariant, the twisted differentials are directly computed
leading to
\begin{gather*}
  d_{Q}\tilde{\va} = d\tilde{\va},\quad d_{Q}\tilde{\vb} =
  d\tilde{\vb},\quad
  d_{Q}\tilde{\psi} = 0,\\
  d_{Q}\varphi = d\varphi + \frac{2}{t^{2}}F =
  2(\varphi\wedge\tilde{\psi} - \tilde{\va}^{T}\wedge\tilde{\vb} +
  \tilde{\Phi}\wedge\tilde{\psi}).
\end{gather*}
In the case that \( S \) is a Lie group with \( \va \), \( \vb \), and
\( d\varphi = 2\omega_{S} \) left-invariant, these results are
constant coefficient with respect to~\( \gamma \).
The coframe
\( \tilde{\delta} = \hat{\delta}/t =
(\tilde{\vA}^{T},\tilde{\vB}^{T},\tilde{\Phi},\tilde{\Psi}) \) is not
\( \tilde{X} \)-invariant, but
\begin{equation*}
  L_{\tilde{X}}\tilde{\delta} = -\tilde{\delta}\ii.
\end{equation*}
When the Kähler form is exact \( \omega_{S}=d\kappa \) with
\( \kappa \) left-invariant and \( S \) is simply-connected, we may
choose a principle parameter~\( \tau \) for~\( C_{0} \) as
in~\S\ref{sec:rotations} and define
\( \delta = \tilde{\delta}\exp(\ii\tau) \).
Then \( X\tau=1 \) implies
\begin{equation*}
  L_{\tilde{X}}\delta = 0.
\end{equation*}
We may now apply the twisted differential~\( d_{Q} \) to the coframe
\( \delta \) to obtain
\begin{equation*}
  \begin{split}
    d_{Q}\delta
    &= d\Bigl(\frac{1}{t}\hat{\delta} \exp(\ii \tau)\Bigr)\\
    &= -\frac{1}{t^{2}}\hat{\psi}\wedge\hat{\delta}exp(\ii\tau) +
      \frac{1}{t}\hat{\delta}\wedge\omega_{\nabla}\exp(\ii\tau) -
      \frac{1}{t}\hat{delta}\wedge \exp(\ii\tau)d\tau\\
    &= - \tilde{\psi}\wedge \delta
      + \delta\wedge \exp(\ii\tau) \omega_{\nabla} \exp(\ii\tau)
      - \delta\wedge\ii d\tau\\
    &= - \tilde{\psi} \wedge\delta
      + \delta\wedge \bigl(\omega_{\LC} + \eta\exp(2\ii\tau)\bigr)
      - \delta\wedge \ii\Bigl(\varphi+\frac{1}{2}\kappa\Bigr).
  \end{split}
\end{equation*}
In the previous expression all non-\( \delta \) terms are constant
coefficient except possibly \( \eta \exp(2\ii \tau) \).
But in the notation of~\S\ref{sec:rotations}, using
Lemma~\ref{lem:eta} we have
\begin{equation*}
  \eta \exp(2\ii \tau) =
  \begin{pmatrix}
    \vP&\vQ&0&0\\
    \vQ&-\vP&0&0\\
    0&0&0&0\\
    0&0&0&0
  \end{pmatrix}
\end{equation*}
which is constant coefficient if and only if \( \vP = \pi^{*}p \) and
\( \vQ = \pi^{*}q \) with \( p,q \) left-invariant.  Therefore,

\begin{theorem}
  The twist of \( T^{*}C \) where \( C \) is the special Kähler cone
  of an invariant projective special Kähler structure on \( S \) of
  dimension~\( 2n \) is a homogeneous quaternionic Kähler manifold of
  dimension \( 4n+4 \).  \qed
\end{theorem}

For the examples we have provided in the paper, the resulting
quaternionic Kähler manifolds are already known, see for example
\cite[Table~2, p.~499]{DeWit-VvP:symmetry}, and the solvable group is
a subgroup of larger isometry group.
The examples may be checked in the same way
in~\cite[\S6]{Macia-S:c-map}, where we verified that the two
structures on \( \CH(1) \), Example~\ref{ex:CH1}, yield
\( \Un(2,2)/(\Un(2)\times \Un(2)) \) and \( G_{2}^{*}/\SO(4) \).
Similarly, the flat cone of Remark~\ref{rem:flat} gives
\( \Un(n,2)/(\Un(n)\times \Un(2)) \); the product
\( \CH(1)^{2} \),~\S\ref{sec:4d}, yields
\( \SO(3,4)/(\SO(3)\times \SO(4)) \) and \( \CH(1)^{3} \),
Proposition~\ref{prop:three}, produces \( \SO(4,4)/(\SO(4)\times \SO(4)) \).

\section*{Acknowledgements}
OM would like to thank the hospitality of the Centre for Quantum
Geometry of Moduli Spaces, and Aarhus University, where part of this
work was developed.
His work was partially supported by the MINECO-FEDER project
MTM2016-77093-P.
AFS was partially supported by the Danish Council for Independent
Research~\textbar~Natural Sciences project DFF - 6108-00358.

\providecommand{\bysame}{\leavevmode\hbox to3em{\hrulefill}\thinspace}
\providecommand{\MR}{\relax\ifhmode\unskip\space\fi MR }
\providecommand{\MRhref}[2]{%
  \href{http://www.ams.org/mathscinet-getitem?mr=#1}{#2}
}
\providecommand{\href}[2]{#2}

{\parindent=0pt\small

\textsc{Oscar Macia}, Departamento de Matematicas, Facultad de
Ciencias Matematicas, Universidad de Valencia, C.~Dr Moliner, 50,
46100 Burjassot, Valencia, Spain. E-mail: \url{oscar.macia@uv.es}

\smallskip \textsc{Andrew Swann}, Department of Mathematics, Centre
for Quantum Geometry of Moduli Spaces, and Aarhus University Centre
for Digitalisation, Big Data and Data Analytics, Aarhus University, Ny
Munkegade 118, Bldg 1530, 8000 Aarhus, Denmark. E-mail:
\url{swann@math.au.dk}
\par}

\end{document}